\documentclass{amsart}

\usepackage[utf8]{inputenc}
\usepackage{amsmath}
\usepackage{amssymb}
\usepackage{latexsym}
\usepackage{amsthm}
\usepackage{bbm}
\usepackage{enumerate}
\usepackage{epsfig}
\usepackage{graphicx}
\usepackage{xspace}
\usepackage{caption}
\usepackage{subfig}
\usepackage[autostyle]{csquotes}
\usepackage{xcolor}
\usepackage{hyperref}
\usepackage{url, hypcap}
\definecolor{darkblue}{RGB}{0,0,160}
\hypersetup{
colorlinks,%
citecolor=darkblue,%
filecolor=black,%
linkcolor=darkblue,%
urlcolor=darkblue
}
\usepackage[
text={440pt,575pt},
headheight=9pt,
centering
]{geometry}

\newtheorem{thm}{Theorem}[section]
\newtheorem{lemma}[thm]{Lemma}

\newtheorem{cor}[thm]{Corollary}

\newtheorem{conj}[thm]{Conjecture}

\theoremstyle{definition}

\newtheorem{remark}[thm]{Remark}
\newtheorem{defn}[thm]{Definition}

\numberwithin{equation}{section}

\newcommand{\RR}{\mathbb{R}}

\newcommand{\Mathematica}{\textsc{Mathematica}\xspace}

\begin{document}

\title[Optimality regions in multiple regression with correlated random coefficients ]{Optimality regions for designs in multiple linear regression models with correlated random coefficients }

\author[U.~Graßhoff]{ Ulrike Graßhoff}
\address[]{Wirtschaftswissenschaftliche Fakultät\\ Humboldt-Universität zu Berlin\\ Unter den Linden 6, 10099 Berlin\\ Germany}
\email{grasshou@hu-berlin.de}
\author[H.~Holling]{Heinz Holling}
\address[]{Westf\"alische Wilhelms-Universit\"at Münster\\ Psychologisches Institut IV\\  Fliednerstr. 21\\ 48149 M\"unster\\ Germany}
\email{holling@uni-muenster.de}
\urladdr{\url{https://www.uni-muenster.de/PsyIFP/AEHolling/de/personen/holling.html}}

\author[F.~Röttger]{Frank Röttger}
\address[]{Fakultät für Mathematik\\ Otto-von-Guericke Universität
Magdeburg\\ Universit\"atsplatz 2, 39106 Magdeburg\\Germany}
\email{frank.roettger@ovgu.de}
\urladdr{\url{http://www.imst3.ovgu.de/Arbeitsgruppe/Frank+Rottger}}
\thanks{Corresponding author: Frank Röttger}
\author[R.~Schwabe]{Rainer Schwabe}
\address[]{Fakultät für Mathematik\\ Otto-von-Guericke Universität
Magdeburg\\ Universit\"atsplatz 2, 39106 Magdeburg\\Germany}
\email{rainer.schwabe@ovgu.de}
\urladdr{\url{http://www.imst3.ovgu.de}}

\begin{abstract}
This paper studies optimal designs for linear regression models with correlated effects for single responses. We introduce the concept of rhombic design to reduce the computational complexity and find a semi-algebraic description for the $ D $-optimality of a rhombic design via the Kiefer--Wolfowitz equivalence theorem. Subsequently, we show that the structure of an optimal rhombic design depends directly on the correlation structure of the random coefficients.

\end{abstract}

\maketitle

\section{Introduction}
Hierarchical regression models with random coefficients enjoy growing importance in biological and psychological applications, whenever there is a variation with respect to the observed subjects. Hereby we often cannot expect the random coefficients to be uncorrelated, which means that we assume that the random coefficients are e.g.~normally distributed with a population mean and a non-diagonal covariance matrix. A special model that will be the topic of this paper are random effects models for linear regression with singular responses, which means that we obtain only one observation per unit. This particular model was motivated by Freund and Holling in \cite{Freund-Holling2008} and Patan and Bogacka in \cite{Patan-Bogacka2007}. 
A natural question that arises is to find optimal experimental designs for these models with respect to some optimality criterion. 
Gra{\ss}hoff et al. determined $ D $-optimal designs that maximize the determinant of the corresponding information matrix, for a couple of different covariance structures in \cite{GHS2009} and \cite{GDHS2012}.
They found that in contrast to fixed effects models for multiple linear regression optimal settings may, surprisingly, occur in the interior of the design region under certain conditions on the covariance structure of the random coefficients. 
In the present paper, we investigate conditions on the covariance structure to discriminate situations in which optimal designs are completely supported on the boundary of the design region as in fixed effects models and situations in which optimal designs may have additional support points in the interior.
This is done for the special class of \emph{rhombic} designs, which are invariant with respect to permutations of the regressors and simultaneous sign change and which we will introduce in Section \ref{s:Multiplelinregression}. Section \ref{s:RhombicEquivalenceTheorem} shows via the Kiefer-Wolfowitz equivalence theorem \cite[Theorem 3.7]{silvey1980optimal} how the parameter regions for which rhombic designs with or without interior points are D-optimal are described by semi-algebraic sets, which are sets defined by polynomial inequalities and equations and how the optimality depends on the covariance structure. Furthermore, we show that for the assumed covariance structure of the random coefficients, the $D$-optimality of designs with interior support points translates to a simple matrix equation for the information. We show as a consequence of the results in Section~\ref{s:RhombicEquivalenceTheorem} that the distinction, whether a $D$-optimal rhombic design requires interior support points or not, can be made by evaluating a polynomial only dependent on the covariance matrix of the random coefficients. Based on these results, we are able to compute optimal designs and their optimality regions explicitly for small to moderate dimensions in Section \ref{s:Examples} and we conjecture results for arbitrary dimensions in Section \ref{s:conjectures}.

\section{General setup}
We consider a random coefficient regression model $ Y_i(x_i)= f(x_i)^{T}b_i + \varepsilon_i, \, i=1,\ldots n $ for observations $Y_i$  at experimental settings $x_i$ where $f$ is a $ p $-dimensional vector of linearly independent regression functions, $b_i$ is a $ p $-dimensional vector of random coefficients and $\varepsilon_i$ are additional observational errors. The random coefficients are assumed to be distributed with unknown mean vector $ \beta $ and prespecified dispersion matrix $ D $, whereas the error terms $\varepsilon_i$ are distributed with zero mean and equal variance $ \sigma_\varepsilon^2 $. Moreover the random coefficients and the error terms are assumed to be uncorrelated.
In this note we assume that all observations $Y_i$ are independent, i.\,e.\ only 
one observation is made for each realization $b_i$ of the random 
coefficients. Moreover, we assume here that an intercept is included in the 
model ($f_1(x)\equiv 1$) such that the additive observational 
error $\epsilon_i$ may be subsumed into the random intercept. This can be achieved by substituting the first entry $ b_{i1} $ in the random coefficient vector by $ b_{i1}+\varepsilon_i $ and the first entry $ d_0 $ in the dispersion matrix D by $ d_0+\sigma_\varepsilon^2 $.
The model can hence be rewritten as a heteroscedastic linear fixed effects model,
\begin{align}
\label{GHS:HETERO}
Y_i(x_i)=f(x_i)^{T}\beta + 
\varepsilon_i ,
\end{align}
where now $ \varepsilon_i=f(x_i)^{T}(b_i-\beta)$  
with mean zero and the variance function defined by 
$\sigma^2(x)=f(x)^{T} D f(x)$. 
Within this heteroscedastic linear model for each single setting $x $ in a design region $ 
\mathcal{X}$ the elemental information matrix \cite{Atkinson2014} equals
$
M(x)=f(x)f(x)^{T}/ 
\sigma^2(x)
$, assuming that $ \sigma^2(x)>0 $ for all $ x \in \mathcal{X} $.
An exact design of sample size n can be described by a finite set of mutually distinct settings $ x_j ,\,  j=1,...,m $, for the explanatory variable and the corresponding numbers $ n_j $ of  replications at $ x_j $, where $ x_j $ may be chosen from the design region $\mathcal{X}$ of potential settings.  Equivalently, a standardized version $ \xi $ may be characterized by the proportions $ \xi(x_j)=n_j/n $ at settings $ x_j $. We call $\xi(x_j)$ the \emph{design weight} at \emph{design point} $ x_j $.
Then for a design $\xi$, the standardized (per observation) information matrix is given by 
$ M(\xi)=\sum_{j=1}^m \xi (x_j)M(x_j), $ which is proportional to the finite sample information matrix with a normalizing constant $ \frac{1}{n} $.
Note that the covariance matrix for the weighted least squares estimator $\hat{\beta}$, which is the best linear unbiased estimator for $\beta$ is proportional to the inverse of the information matrix. Hence, maximizing the information matrix is equivalent to minimizing the 
covariance matrix of $\hat{\beta}$. 

To compare different designs we consider the most popular criterion, the $D$-criterion, with 
respect to which a design $\xi$ is $D$-optimal, if it maximizes the logarithm $ \log \det(M(\xi)) $ of the determinant of the information matrix. This is equivalent to the minimization of the volume of the confidence ellipsoid for $\beta$ under the assumption of normally distributed errors.  As discrete optimization on the set of exact designs is generally to complicated we relax the condition on the weights $ \xi(x_j) $ being multiples of $ \frac{1}{n} $ and consider approximate designs $ \xi $ in the spirit of Kiefer \cite{Kiefer1974} with real-valued weights $ w_j=\xi(x_j)\ge 0 $ satisfying $ \sum_{j=1}^{n}w_j=1. $ A detailed introduction to the theory of optimal design is \cite{silvey1980optimal}.
In the setting of approximate designs, for which the proportions 
$\xi(x)$ are not necessarily multiples of $1/n$, where $n$ denotes the 
sample size, the $D$-optimality of a design $\xi^*$ can be established by the 
well-known Kiefer-Wolfowitz equivalence theorem (see \cite[Theorem 2.2.1]{Fedorov1972}, for a 
suitable version):
\begin{thm}[Extended Kiefer-Wolfowitz equivalence theorem]\label{t:equivalence}
A design $\xi^*$ is $D$-optimal on $ \mathcal{X} $, if and only if $f(x)^{T} 
M(\xi^*)^{-1}f(x)/ \sigma^2(x)\leq p$, 
uniformly in $x \in \mathcal{X}$.
\end{thm}
Let $ \psi(x;\xi):=f(x)^{T} (p D- 
M(\xi)^{-1})f(x) $ be the suitably transformed sensitivity function. When we substitute 
$\sigma^2(x)=f(x)^{T} Df(x)$ 
into this relation and rearrange terms, $D$-optimality is achieved, if
\begin{equation}
\label{GHS:EQUI1}
\psi(x;\xi^*)\geq 0
\end{equation}
for all $x\in \mathcal{X}$.
Moreover, equality is attained in \eqref{GHS:EQUI1} for design points in the support of an optimal design $ \xi^* $:
\begin{cor}[{\cite[Corollary~3.10]{silvey1980optimal}}]\label{c:silvey}
	It holds that $ \psi(x;\xi^*)=0 $ for all $ x\in \mathcal{X} $ with $ \xi^*(x)>0 $.
\end{cor}
For notational convenience we define
\[\Gamma(\xi):=p D- M(\xi)^{-1},\]
such that $ \psi(x;\xi)=f(x)^{T} \Gamma(\xi)f(x) $.

\section{Multiple linear regression}\label{s:Multiplelinregression}
In the following we consider the situation of a multiple linear regression model with $K$ factors where we have $ n $ observations
\begin{align}\label{eq:model}
	Y_i(x_i)&=\beta_0+\sum_{k=1}^K \beta_k x_{ik}+\varepsilon_i  
\end{align}
with $x_i=(x_{i1},...,x_{iK})^{T} \in \mathcal{X}=[-1,1]^K$ and $ \text{Var}(\varepsilon_i)=\sigma^2(x_i)=f(x_i)^T D f(x_i) $. Here we assume that we can choose the design points from the symmetric standard hypercube. The vector of regression functions is given by $f(x)=(1,x_1,...,x_K)^{T}$, such that the model contains an intercept by the first component of $ f $. Note that now $ p=K+1 $ and $ \beta=(\beta_0,\ldots,\beta_K)^T $.

We assume that the random coefficients $ b_{i1},...,b_{iK} $ associated with the components $ x_1,...,x_K $ of the regressor are homoscedastic with variance $ d_1 $ and equi-correlated with covariance $ d_2 $.
Moreover, let the random intercept $ b_{i0} $ be uncorrelated with the other random coefficients. To be more precise we consider a $ p\times p $-dimensional dispersion matrix
\begin{align}
D& =\begin{pmatrix}
d_0&&0\\
0&&D_1\\
\end{pmatrix},\label{eq:D}
\end{align}
where $ D_1 $ is a completely symmetric $ K \times K $-dimensional matrix
\begin{align*}
D_1&= (d_1-d_2)I_{K}+ d_2 \mathbbm{1}_{K}\mathbbm{1}_{K}^T.
\end{align*}
Here, $ I_k $ defines the $ K\times K $-dimensional identity matrix and $ \mathbbm{1}_{K} $ the vector of length $ K $ where all entries equal $ 1 $.

\begin{defn}[Model cone]
	We define
	\[\mathcal{C}_K:=\left\{(d_0,d_1,d_2)^T \in \mathbb{R}^3~\middle| ~d_0>0,d_1>0, -\frac{d_1}{K-1}\le d_2\le d_1\right\}\]
	as the \emph{model cone}, so the values of $ (d_0,d_1,d_2)^T $ where $ D $ is non-negative definite.
\end{defn}

\subsection{Diagonal dispersion matrix}\label{s:diagonalD}
To start we first assume additionally that $ d_2=0 $ which means that all components of the random coefficients are uncorrelated, so that the dispersion matrix
\[D=\begin{pmatrix}
d_0&& 0\\
0 && d_1 I_K\\
\end{pmatrix}\] 
of the random coefficients $b_i$ is a diagonal matrix.
Hence, the variance of each design point is equal to $\sigma^2(x)=d_0+ d_1 \sum_{j=1}^K x_j^2$. 
In \cite{GDHS2012} it is shown that uniform full factorial $2^K$-designs supported on the points $(\pm x_1,....,\pm x_K)$ are $D$-optimal. It holds that $ x_1=\ldots= x_K=x^* $ is optimal and it depends on the values of $ d_0 $ and $ d_1 $ if $ x^*=1 $ or $ x^*<1 $. The designs constitute the orbit generated by $(x_1,....,x_K)$ with respect to the (finite) group of transformations of both sign changes within the factors and permutations of the factors themselves. 
For more details see \cite{GDHS2012}. 

\subsection{Non-diagonal dispersion matrix}
We now assume that $ d_2\neq 0 $.  
To reduce the complexity of the system of polynomial equations and inequalities given by the equivalence theorem, one can apply various methods. One of these methods is to assume a certain design structure. A standard approach is the assumption of symmetry in the design under some group action and a restriction of the design region. This is motivated from the symmetry of $ D $-optimal designs for the situation with $ d_2=0 $ as described above in Section \ref{s:diagonalD}. This applies as the $ D $-criterion $ \log \det (M(\xi)) $ is not affected by these transformations $ g $  as $ \det M(\xi^g) = \det M(\xi) $.
Hence, by convexity the class of invariant designs constitutes an essentially complete class such that search may be restricted to invariant designs.
A particular class of invariant designs are rhombic designs. Let $ \text{Sym}(K) $ denote the permutation group on $ K $ elements and $ \{\pm 1\} $ the permutation group with respect to a global sign change, which is therefore isomorphic to Sym$ (2) $.
\begin{defn}
Let a design $ \xi $ for the given model with support on the space diagonals without the origin and at most two points per space diagonal be a \emph{rhombic design} if it is invariant under the action of $ \text{Sym}(K)\times \{\pm 1\} $ on the design points. 
\end{defn}
This means that we study designs on $ [-1,1]^{K} $, that are invariant under permutations among the entries of each design point and a global sign change with support on the diagonals of maximal length without the origin. There are $ \lfloor \frac{K}{2}\rfloor+1 $  different orbits under this group action, where $ \lfloor z \rfloor $ denotes the integer part of some $ z \in \RR $. We define $ \tilde{K}:= \lfloor \frac{K}{2}\rfloor $. Now, let $ x_\ell $ for $ \ell \in \{ 0,1,\ldots,\tilde{K} \} $ denote the location parameter of each orbit $ \mathcal{O}_\ell(x_\ell) $ with $ 0< x_\ell \le 1 $ and either $ \ell $ or $ K-\ell $ negative signs. The location parameter denotes the absolute value of the entries of the design points in $ \mathcal{O}_\ell(x_\ell) $ as the design points of a rhombic design are restricted to the space diagonals. Let $ N_\ell=2 \binom{K}{\ell} $ for $ \ell \neq \frac{K}{2} $ and  $ N_\ell=\binom{K}{\ell} $ for $ \ell=\frac{K}{2} $. 

Rhombic designs can be characterized as follows: Let $ x_j $ for $ j \in \{0,1,\ldots,\tilde{K}\} $ be a design point with entries of the same absolute value, i.e. $ x_j $ lies on a space diagonal of $ \mathcal{X} $. Let $ \mathcal{O}_j $ be the orbit of $ x_j $ under the action of $ \text{Sym}(K)\times \{\pm 1\} $ and let $ \bar{\xi}_j $ be the uniform design on $ \mathcal{O}_j $ that assigns the proportion $ \frac{1}{N_j} $ to each $ x\in \mathcal{O}_j $. If every orbit $ \mathcal{O}_j $ is attributed with a weight $ w_j \ge 0 $ such that $ \sum_{j=0}^{\tilde{K}}w_j=1 $, then $ \sum_{j=0}^{\tilde{K}} w_j \bar{\xi}_j $ is a rhombic design.    

To formalize the invariance considerations, let $ g $ denote the group action that generates rhombic designs. Then, there exists a matrix $ Q_g $ so that $ f(g(x))=Q_gf(x) $.

Figures \ref{Rhombic2D} and \ref{Rhombic3D} exemplify the two different rhombic design classes that will be studied separately. This distinction is made on the location of the design points. With rhombic vertex designs we refer to rhombic designs, where the support is restricted to the vertices of the hypercube, while non-vertex designs are allowed to have points on both the vertices and the interior or in the interior only. In Figure \ref{Rhombic2D}, the blue points denote the orbit of $ (x_0,x_0) $, with $ 0<x_0\le 1 $, while the red points denote the orbit of $ (x_1,-x_1) $, with $ 0<x_1\le 1 $. Similarly, in Figure~\ref{Rhombic3D}, the blue points denote the orbit of $ (x_0,x_0,x_0) $, with $ 0<x_0\le 1 $, while the red points denote the orbit of $ (x_1,x_1,-x_1) $, with $ 0<x_1\le 1 $. We chose the name \emph{rhombic designs} due to the structure of the design points in Figure \ref{Rhombic2D}.

\begin{figure}
\centering
\subfloat[Vertex rhombic design]{
\includegraphics[scale=0.45]{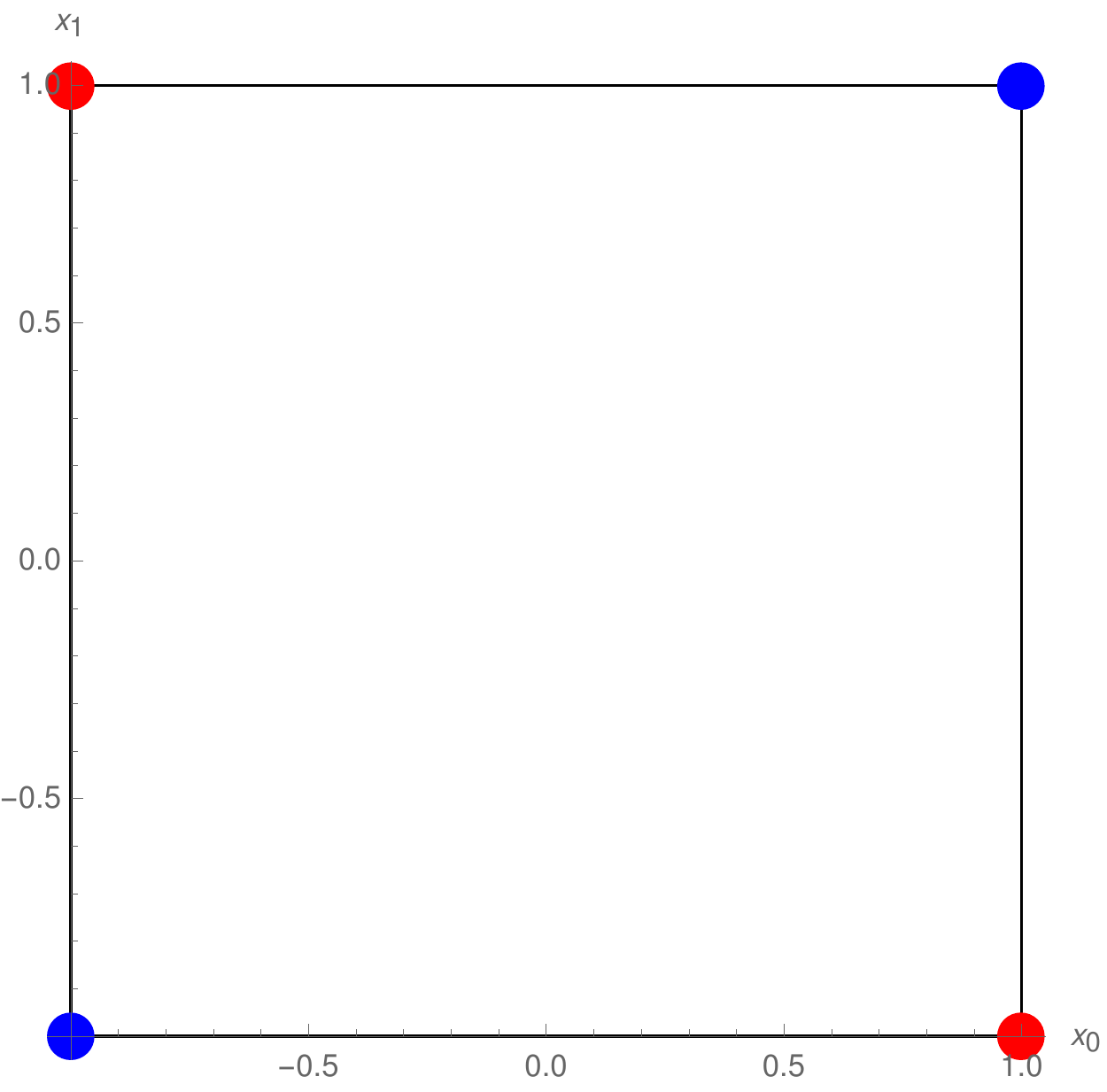}}
~~~~~~~~~~~\hspace{1cm} \subfloat[Non-vertex rhombic design]{
	\includegraphics[scale=0.45]{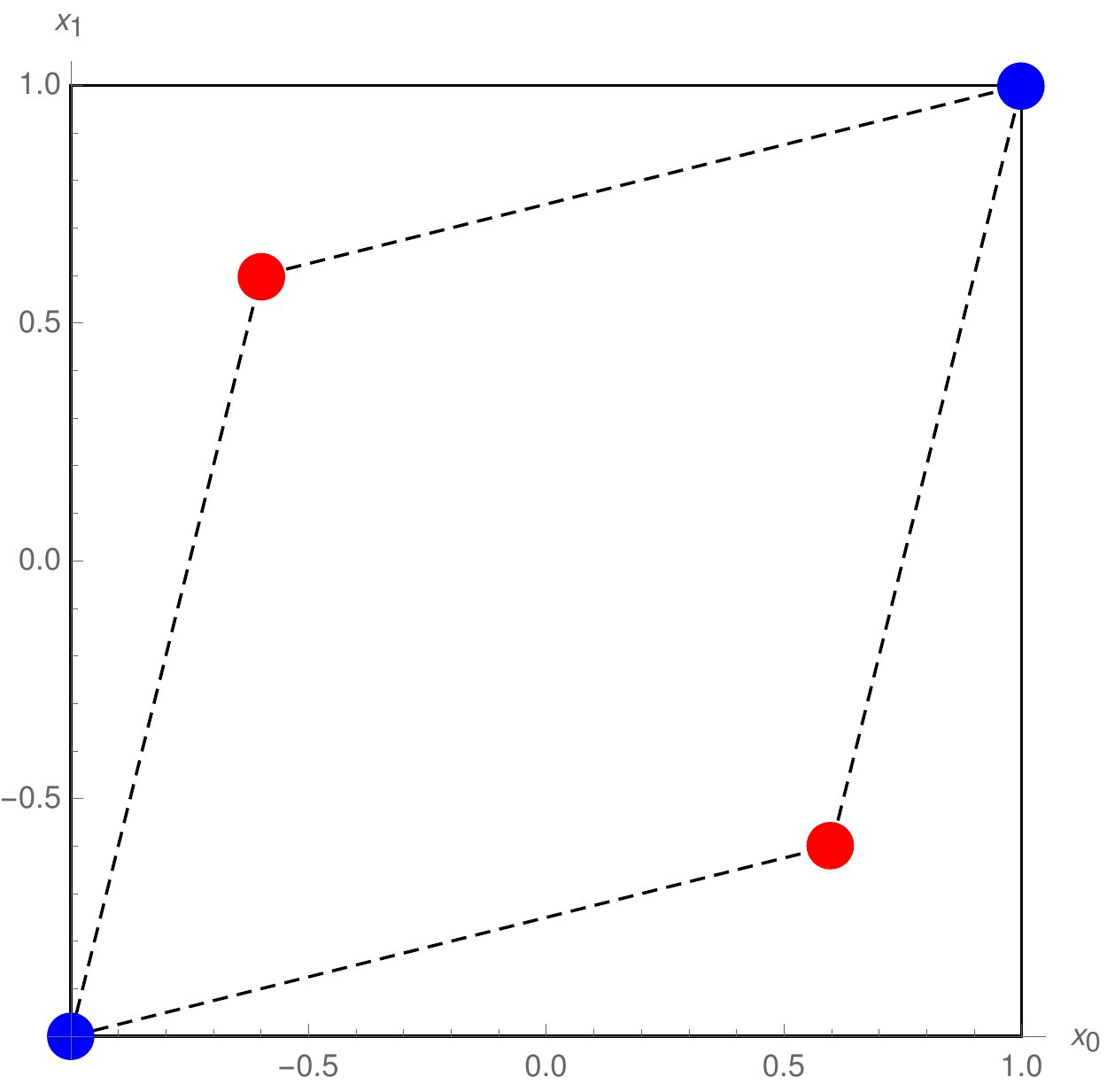}}
\caption{Examples for design points for $ K=2 $.  }\label{Rhombic2D}
\end{figure}

\begin{figure}
\centering
\subfloat[Vertex rhombic design]{\includegraphics[scale=0.45]{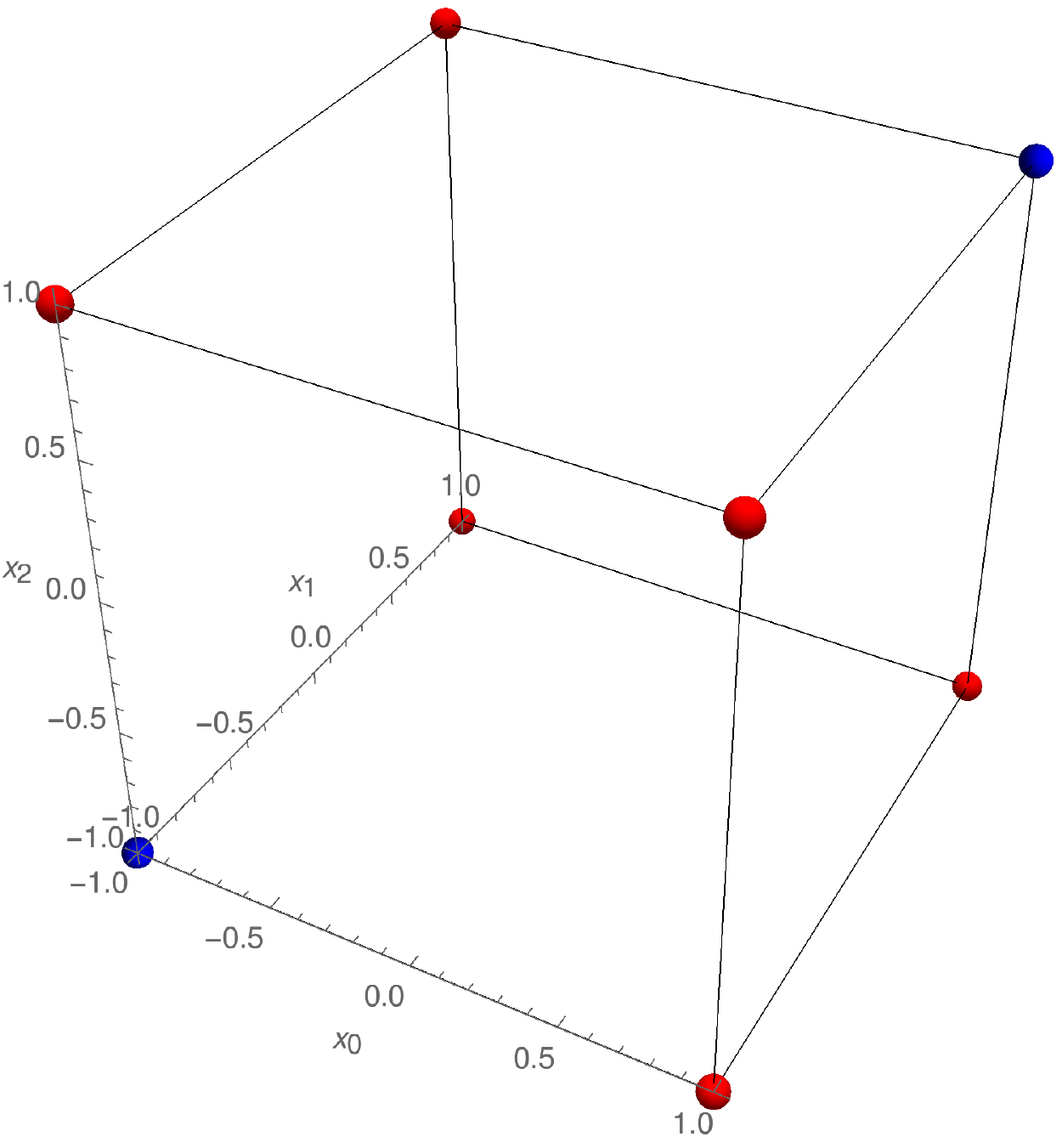}}
~~~~~\hspace{1cm} \subfloat[Non-vertex rhombic design]{
	\includegraphics[scale=0.45]{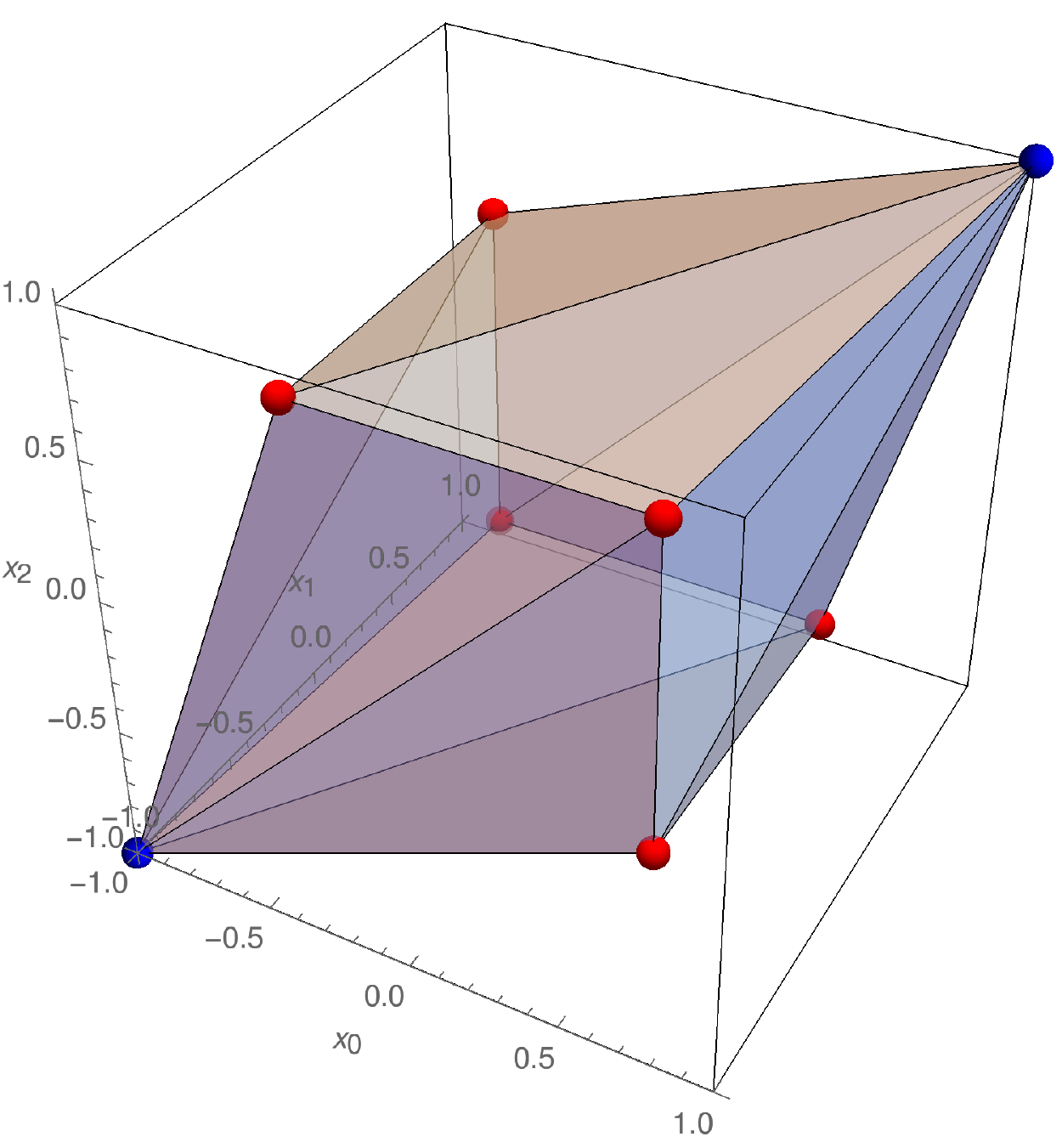}}
\caption{Examples for design points for $ K=3 $. }\label{Rhombic3D}

\end{figure}
The usefulness of rhombic design is mainly due to the complexity reduction that comes from its definition: Instead of finding $ K+1 $ design points each with $ K $ location variables for the entries and a variable for the design weight, we restrict the problem to $ \tilde{K}+1 $ orbits with one weight variable and one location variable per orbit.

\begin{lemma}
The variance function $ \sigma^2(x)=f(x)^T D f(x) $ is equal for all $ x $ in one orbit. 
\end{lemma}
\begin{proof}
Let $ g $ be the group action from above that generates rhombic designs.
By the form of $ D $ in (\ref{eq:D}), it holds for $ \sigma^{2}(x) $ that
\begin{align*}
\sigma^2(g(x))&=f(g(x))^T D f(g(x))\nonumber \\ 
&=f(x)^T Q_g^T D Q_g f(x).
\end{align*}
  As $ Q_g^T D Q_g =D $, it follows that $ \sigma^2(x)=\sigma^2(g(x)) $.
\end{proof}

The information matrix of a rhombic design is
\begin{align}\label{eq:information}
M(\xi)=\sum_{\ell =0}^{\tilde{K}}\frac{w_\ell }{N_\ell \sigma^2(x_\ell )}\sum_{x \in \mathcal{O}_\ell }f(x)f(x)^T.
\end{align}
To compute the matrix $\sum_{x \in \mathcal{O}_\ell }f(x)f(x)^T $, see that the information matrix is for each orbit structured into a scalar entry in the upper left corner and a $ K\times K $-dimensional lower right completely symmetric block matrix. This leads to the following Lemma:
\begin{lemma}\label{l:InfMatrix}
In the setting of Section \ref{s:Multiplelinregression}, a rhombic design $ \xi $ has an information matrix of the form
\[M(\xi)=\begin{pmatrix}
m_0(\xi)&&0\\
0&&M_1(\xi)\\
\end{pmatrix},\]
where $ m_0(\xi)=\sum_{\ell=0}^{\tilde{K}}\frac{w_\ell}{\sigma^2(x_\ell)} $ and $ M_1(\xi)=\sum_{\ell=0}^{\tilde{K}}\frac{w_\ell}{N_\ell\sigma^2(x_\ell)}\sum_{x \in \mathcal{O}_\ell}xx^T= (m_1(\xi)-m_2(\xi))I_K+m_2(\xi) \mathbbm{1}_{K}\mathbbm{1}_{K}^T $ is a completely symmetric $ K\times K $-dimensional matrix.
\end{lemma}
\begin{proof}
As for every $ x \in [-1,1]^{K} $, $ -x $ lies in the same orbit as $ x $. Therefore, as
\begin{align*}
f(x)f(x)^T+f(-x)f(-x)^T &=2\begin{pmatrix}
1&&0\\
0&&xx^T\\
\end{pmatrix},
\end{align*}
$ M(\xi) $ is of the form
\[M(\xi)=\begin{pmatrix}
m_0(\xi)&&0\\
0&&M_1(\xi)\\
\end{pmatrix},\]
where 
\[ M_1(\xi)=\sum_{\ell=0}^{\tilde{K}}\frac{w_\ell}{N_\ell\sigma^2(x_\ell)}\sum_{x \in \mathcal{O}_\ell}xx^T \]
is a $ K\times K $-dimensional symmetric matrix and
\[m_0(\xi)=\sum_{\ell=0}^{\tilde{K}}\frac{w_\ell}{\sigma^2(x_\ell)}.\]
Now, $ \sum_{x \in \mathcal{O}_\ell}xx^T $ is completely symmetric,because of the permutation invariance of the orbit $ \mathcal{O}_\ell $. As the (weighted) sum of completely symmetric matrices is completely symmetric itself, the Lemma follows.
\end{proof}

We defined $ \Gamma(\xi) $ as the matrix in the quadratic form in (\ref{GHS:EQUI1}) coming from the equivalence theorem. By Lemma \ref{l:InfMatrix}, denoting the lower right $ K\times K $-submatrix of $ D $ by $ D_1 $,
\[\Gamma(\xi)=\begin{pmatrix}
(K+1)d_0-m_0(\xi)^{-1}&&0\\
0&&(K+1)D_1-M_1(\xi)^{-1}\\
\end{pmatrix}. \]
$ \Gamma(\xi) $ has the same block structure as $ D $ and $ M(\xi) $ with completely symmetric lower right block of dimension $ K\times K $, so
\begin{align*}\label{def:Gamma}
\Gamma(\xi) &= \begin{pmatrix}
\gamma_0&&0\\
0&&(\gamma_1-\gamma_2)I_K+\gamma_2 \mathbbm{1}_{K}\mathbbm{1}_{K}^T\\
\end{pmatrix}
\end{align*}
with $ \gamma_0=(K+1)d_0-m_0(\xi)^{-1}. $

We remind the reader of the following well-known fact:

\noindent
The inverse of a completely symmetric $ K\times K $ matrix $ A $ with
	\begin{align*}
	A=&(a_1-a_2)I_K+a_2 \mathbbm{1}_{K}\mathbbm{1}_{K}^T \\
	\intertext{is}
	A^{-1}=&\frac{1}{a_1-a_2}I_K-\frac{a_2}{(a_1-a_2)(a_1+(K-1)a_2)} \mathbbm{1}_{K}\mathbbm{1}_{K}^T. 
	\end{align*}

With these preparatory results, we are able to investigate the optimality of rhombic designs.

\section{Rhombic Designs and the Equivalence Theorem}\label{s:RhombicEquivalenceTheorem}

\subsection{Rhombic Vertex Designs}

This section studies rhombic designs with all design points on the vertices of the hypercube. We will use the Kiefer-Wolfowitz equivalence theorem to investigate how the $ D $-optimality of a rhombic vertex design depends on $ D $. 
The investigation leads to the following theorem:
\begin{thm}\label{t:vertex}
For the given model \ref{eq:model} with a dispersion matrix as defined in \ref{eq:D}, let $ \xi^* $ be a rhombic vertex design. If either
\begin{enumerate}
	\item $ K $ is even or
	\item $ K $ is odd and $ \xi^* $ has support not only on $ \mathcal{O}_{\tilde{K}} $,
\end{enumerate} 
then $ \xi^* $ is $ D $-optimal if and only if the matrix
$ \Gamma(\xi^*)=pD-M(\xi^*)^{-1} $
is a diagonal matrix
\[\Gamma(\xi^*)=\begin{pmatrix}
\gamma_0 && 0 \\
0 && \gamma_1 I_K\\
\end{pmatrix},\]
with $ \gamma_0\ge 0 $ and $ \gamma_1=-\frac{\gamma_0}{K} $. 
\end{thm}
\begin{proof}
$ "\Leftarrow" $
Assuming $ \gamma_0\ge 0 $, $ \gamma_1=-\frac{\gamma_0}{K} $ and $ \gamma_2=0 $, it follows that 
\begin{align*}
\psi(x;\xi^*)&=f(x)^T\Gamma(\xi^*)f(x)\\
&=\gamma_0+\gamma_1||x||^2\\
&=\gamma_0\left(1-\frac{||x||^2}{K}\right)
\end{align*}
Hence $ \psi(x;\xi^*)\ge 0 $ for all $ x \in [-1,1]^{K} $, as $ ||x||^2 \le K $. Therefore, the $ D $-optimality follows from  Theorem~\ref{t:equivalence}.

\noindent
$ "\Rightarrow" $
According to (\ref{GHS:EQUI1}) in Theorem \ref{t:equivalence}, a design $ \xi^* $ is $ D $-optimal if and only if 
\begin{align}\label{eq:directionalderivative}
\psi(x;\xi^*)&=f(x)^T \Gamma(\xi^*) f(x)\ge 0 
\end{align}
for all $ x \in \mathcal{X} $. Furthermore, by Corollary \ref{c:silvey}, we know that
\begin{align}
\psi(x;\xi^*)&= 0 \label{eq:vertexequivalence}
\end{align}
for all support points of $ \xi^* $.
Now, by (\ref{eq:vertexequivalence}), it follows for any design point $ x_\ell \in \mathcal{O}_\ell(1) $, that
\begin{align} \label{eq:directionalderivative2}
\psi(x_\ell,\xi^*)&=\gamma_0+K\gamma_1+(K(K-1)-4\ell(K-\ell)) \gamma_2=0.
\end{align}
Assuming that the design is supported on at least two different orbits, this directly implies that $ \gamma_2 $ equals zero, as $ (K(K-1)-4\ell(K-\ell)) $ is strictly monotone for $ 0\le \ell \le \tilde{K} $.

Now, say that $ \xi^* $ is only supported on a single orbit $ \mathcal{O}_\ell $ with $ 0\le \ell \le \tilde{K} $ for even $ K $ and $ 0\le \ell < \tilde{K} $ for odd $ K $. If $ \ell=0 $, it is easy to see that the information matrix is singular, therefore such a design cannot be $ D $-optimal. The same is true for even $ K $ when $ \ell=\tilde{K} $ using the same argument as in \cite{Freise2020}.
It holds that $ \psi(x_\ell,\xi^*)=0 $ for all $ x_\ell\in \mathcal{O}_\ell(1)  $. From the $ D $-optimality of $ \xi^* $ it follows that $ \psi(x_{\ell-1},\xi^*)\ge 0 $ for $ x_{\ell-1}\in \mathcal{O}_{\ell-1}(1) $ and  $ \psi(x_{\ell+1},\xi^*)\ge 0 $ for $ x_{\ell+1}\in \mathcal{O}_{\ell+1}(1) $, so in the orbits with one less or one more negative entry in the vector. If $ \gamma_2\neq 0 $, this would imply  $ \psi(x_{\ell-1},\xi^*)>\psi(x_{\ell},\xi^*)>\psi(x_{\ell+1},\xi^*) $ or  $ \psi(x_{\ell+1},\xi^*)>\psi(x_{\ell},\xi^*)>\psi(x_{\ell-1},\xi^*) $. This contradicts the assumed $ D $-optimality of $ \xi^* $, therefore $ \gamma_2=0 $ holds.

It follows from (\ref{eq:directionalderivative2}) and $ \gamma_2=0 $ that
\begin{align*}
\gamma_1&=-\frac{\gamma_0}{K}.
\end{align*}
This implies that
\[ \psi(x;\xi^*)=\gamma_0\left(1-\frac{||x||^2}{K}\right) .\]
and therefore $ \gamma_0\ge 0 $ when the design on the vertices is $ D $-optimal.
\end{proof}

Note that $ \psi(x,\xi^*)=0 $ for all vertices $ x $ of the hypercube as for those we have $ ||x||^2=K $.

\begin{cor}\label{c:vertex}
A rhombic vertex design with support on either at least two orbits or an orbit $ \mathcal{O}_\ell(1) $ with $ 1\le \ell <\tilde{K} $ can only be $D$-optimal when
\[(d_1-d_2)(d_1+(K-1)d_2)-d_0(d_1+(K-2)d_2)\le 0\]
\end{cor}

\begin{proof}
	From Theorem \ref{t:vertex} we obtain that
	\begin{align*}
	\gamma_0\ge 0,~~~~~	\gamma_1&=-\frac{\gamma_0}{K},~~~~~ \gamma_2=0 
	\end{align*} 
	is equivalent to the $D$-optimality of a rhombic vertex design with support on at least two orbits or an orbit $ \mathcal{O}_\ell(1) $ with $ 1\le \ell <\tilde{K} $. The equation system $ \{\gamma_1=-\frac{\gamma_0}{K},~~~~~ \gamma_2=0\} $ has two solutions $ m_0(\xi)^{\pm} $ for $ m_0(\xi) $ in dependence of $ d_0,d_1,d_2 $ and $ p $ that we obtained with \Mathematica:
	\begin{multline*}
	m_0^{\pm}(\xi)=\frac{d_0 (p+1)+(p-1) \left(d_1 p+d_1+d_2 p^2-3 d_2 p\right)}{2 p \left(d_0^2+d_0 (p-1) (2 d_1+d_2 (p-3))+(p-1)^2 (d_1-d_2) (d_1+d_2 (p-2))\right)}\\ \pm   \frac{(p-1) \sqrt{d_0^2+2 d_0 (d_1 (p-1)+d_2 (p-3) p)+(p-1) \left(d_1^2 (p-1)+2 d_1 d_2 (p-3) p+d_2^2 p \left(p^2-5 p+8\right)\right)}}{2 p \left(d_0^2+d_0 (p-1) (2 d_1+d_2 (p-3))+(p-1)^2 (d_1-d_2) (d_1+d_2 (p-2))\right)}.
	\end{multline*}
	Now, with $ \gamma_0= p d_0 - \frac{1}{m_0(\xi)} $ we see that
	\begin{align*}
		\Leftrightarrow 0&\ge\left(p d_0 - \frac{1}{m_0^+(\xi)}\right)\left(p d_0 - \frac{1}{m_0^-(\xi)}\right)\\
	\Leftrightarrow 0	&\ge p^2d_0^2m_0^+(\xi)m_0^-(\xi)-(m_0^+(\xi)+m_0^-(\xi))+1\\
	\Leftrightarrow 0&\ge (d_1-d_2)(d_1+(K-1)d_2)-d_0(d_1+(K-2)d_2).
	\end{align*} 
To derive the corollary, check that $ p d_0 - \frac{1}{m_0^-(\xi)} $ is always negative on $ \mathcal{C}_K $, so that 
\begin{align*}
0&\le p d_0 - \frac{1}{m_0(\xi)}\\
\Leftrightarrow 0&\le p d_0 - \frac{1}{m_0^+(\xi)}\\
\Leftrightarrow 0&\ge \left(p d_0 - \frac{1}{m_0^+(\xi)}\right)\left(p d_0 - \frac{1}{m_0^-(\xi)}\right)
\end{align*}
on $ \mathcal{C}_K $. This implies the corollary.
\end{proof}

\begin{remark}\label{r:vertex}
Theorem \ref{t:vertex} gives a semi-algebraic description of the optimality region of rhombic vertex design in the design weights and the coefficients $ d_0,d_1 $ and $ d_2 $. This means that 
\[ \left\{\gamma_0\ge 0, \gamma_1 =-\frac{\gamma_0}{K}, \gamma_2=0\right\} \] 
can be interpreted as a semi-algebraic set if one takes into account the constraints that $ \xi $ is a design and $ (d_0,d_1,d_2)^T \in \mathcal{C}_K $.
This allows us to obtain symbolic solutions for the design weights $ \xi^*_\ell:=\xi^*(\mathcal{O}_\ell(1)) $ in dependence of the coefficients $ d_0,d_1 $ and $ d_2 $.
\end{remark}

\subsection{Non-vertex (rhombic) designs}

Instead of restricting to rhombic designs we will discuss a broader class of designs, namely designs with a design point in the interior of the hypercube. This design class naturally includes non-vertex rhombic designs.  

\begin{thm}\label{t:non-vertex}
A design $ \xi^* $ with at least one design point in the interior of the hypercube is $ D $-optimal if and only if $ M(\xi^*)=\frac{1}{K+1}D^{-1} $, which means that  $ \Gamma(\xi^*)=(K+1)D-M(\xi^*)^{-1} $ is zero.
\end{thm}
\begin{proof}
$ "\Leftarrow"	 $ $ M(\xi^*)=\frac{1}{K+1}D^{-1} $ implies the $ D $-optimality of $ \xi^* $ by Theorem~\ref{t:equivalence}.

\noindent $ "\Rightarrow" $	
By Theorem \ref{t:equivalence} and Corollary \ref{c:silvey}, $ \xi^* $ is $ D $-optimal if and only if
\[\psi(x;\xi^*)=f(x)^T((K+1)D-M(\xi^*)^{-1})f(x)\ge 0 \]
for all $ x \in [-1,1]^{K} $ and 
\[\psi(x;\xi^*)= 0, \]
if $ x $ is a design point of $ \xi^* $. Now, with $  \Gamma(\xi^*)=(K+1)D-M(\xi^*)^{-1} $,
\begin{align*}
\psi(x;\xi^*)=f(x)^T\Gamma(\xi^*) f(x) &=\gamma_0+ \gamma_1 ||x||^2 + 2\gamma_2 \sum_{1\le \ell<j\le K}x_\ell x_j
\end{align*}
is a quadratic polynomial. It holds that
\begin{align}
\gamma_0+ \gamma_1 ||x||^2 + 2\gamma_2 \sum_{1\le \ell<j\le K}x_\ell x_j&=0 \label{eq:non-vertex}
\end{align}
for design points $ x $ of $ \xi^* $. 
As $ \xi^* $ is a non-vertex design there is an interior point $ x_1 $ and additionally K further design points $ x_2,...x_{K+1} $ such that $ x_1,...,x_{K+1} $ span $ \RR^K $ because the information matrix $ M(\xi^*) $ needs to be non-singular. Fix the affine subspace generated by $ x_1 $ and $ x_2 $. It holds that $ \psi(x_1,\xi^ *)=\psi(x_2,\xi^ *)=0 $ and $ \psi(\lambda x_1+(1-\lambda)x_2,\xi^ *)\ge 0 $ for all $ \lambda \in [0,1] $ and additionally for some $ \lambda<0 $ because $ x_1 $ is an interior point of $ [-1,1]^K $. On the affine subspace generated by $ x_1 $ and $ x_2 $, $ \psi(x;\xi^ *) $ is a quadratic polynomial in $ \lambda $. The only quadratic polynomial that is zero on at least two points and non-negative on at least one point on  the line segment between these points as well as for at least one point on the line outside this segment is the zero polynomial. Recursively, this can be extended to higher dimensions. Therefore,
\begin{align*}
f(x)^T((K+1)D-M(\xi^*)^{-1})f(x)\ge 0
\end{align*}
for all $ x\in [-1,1]^K $ can only be achieved as an equation, so 
\begin{align*}
(K+1)D-M(\xi^*)^{-1}=0.
\end{align*}
Hence, the Theorem follows.
\end{proof}

\begin{cor}\label{c:non-vertex}
An invariant design $ \xi^* $ with a design point in the interior of the hypercube can only be $D$-optimal if the first diagonal entry of $ D^{-1} $ is larger than the second, which means that 
\[(d_1-d_2)(d_1+(K-1)d_2)-d_0(d_1+(K-2)d_2)>0 .\]
\end{cor}

\begin{proof}
	According to Theorem \ref{t:non-vertex}, an invariant design $ \xi^* $ with a design point in the interior of the hypercube is $D$-optimal if and only if $ \Gamma(\xi^*)=0 $.
	Now, this implies that $ M(\xi^*)=\frac{1}{K+1}D^{-1} $ and therefore
	\begin{align*}
		m_0(\xi^*)&=\frac{1}{(K+1)d_0}, &&m_1(\xi^*)=\frac{d_1+(K-2)d_2}{(K+1)(d_1-d_2)(d_1+(K-1)d_2)},
	\end{align*}
	where $ m_1(\xi^*) $ denotes the diagonal entries of $ M_1(\xi^*) $.
	 It is easy to see that $ m_0(\xi^*)>m_1(\xi^*) $ for all designs with interior design points, so we obtain 
	 \begin{align*}
	 	&m_0(\xi^*)>m_1(\xi^*)\\
	 	\Leftrightarrow~~~ &(d_1-d_2)(d_1+(K-1)d_2)-d_0(d_1+(K-2)d_2)>0.
	 \end{align*}
\end{proof}

\begin{remark} \label{r:non-vertex}
Theorem \ref{t:non-vertex} describes the semi-algebraic structure of the optimality area of non-vertex designs. We see that this structure is given by the non-negative real part of an algebraic variety, so the vanishing set of a collection of polynomials under the constraints of the model cone $ \mathcal{C}_{K} $ and the design simplex.
This means that we can obtain symbolic solutions for the optimal designs weights and design points in dependence of the coefficients $ d_0,d_1 $ and $ d_2 $ by studying the set $ \left\{\Gamma(\xi)=0\right\} $ under the imposed constraints.
\end{remark}

\section{Rhombic Designs for \texorpdfstring{$  K \in \{2,3,4,5 \} $}{K in (2,3,4,5)}}\label{s:Examples}

This section investigates for which values $ d_0, d_1 $ and $ d_2 $ we find a vertex or non-vertex rhombic design for $ 2\le K\le 5 $.
\subsection{The case \texorpdfstring{$ K=2 $}{K=2}}
The results of this section where first calculated by hand and later confirmed with a \Mathematica implementation of Theorems \ref{t:vertex} and \ref{t:non-vertex}. For $ K=2 $, it is
\begin{align}\label{eq:Dforp=3}
D=\begin{pmatrix}
d_0 &0 & 0\\
0 & d_1 &  d_2 \\
0 &  d_2  & d_1\\
\end{pmatrix}
\end{align}
where $ |d_2|<d_1 $ and the variance of each design point is equal to 
\[
\sigma^2(x)=f(x)^{T}\, D\, f(x)=d_0 +d_1(x_0^2+x_1^2)+2d_2 x_0 x_1.
\]
The symmetric properties of the covariance structure with respect to the random effects of the two attributes, $\text{Var}(b_1)=\text{Var}(b_2)=d_1$, motivates us to consider as candidates for the $D$-optimal designs the following rhombic designs $\xi_{x_0\Diamond x_1,w}$ consisting  of the four design points $(x_0,x_0)$, $(-x_0,-x_0)$, $(-x_1,x_1)$ and $(x_1,-x_1)$ for $x_0,x_1 \in [-1,1]$ which form a centered rhombus within the design region. Therefore we have $ \mathcal{O}_0(x_0)=\{(x_0,x_0),(-x_0,-x_0)\} $ and $ \mathcal{O}_1(x_1)=\{(-x_1,x_1),(x_1,-x_1)\} $, so that $\sigma^2(x_0,x_0)=\sigma^2(-x_0,-x_0)$ and $\sigma^2(-x_1,x_1)=\sigma^2(x_1,-x_1)$. It follows that we can deal with two distinct weights $w_0=\xi(\mathcal{O}_0(x_0))$ and $w_1=\xi(\mathcal{O}_1(x_1))$ Since the sum of the weights of all orbits is equal to $1$ we can set $w_0=w$ and $w_1=1-w$.
The information matrix for $\xi_{x_0\Diamond x_1,w}$ results in 
\begin{align*}
M(\xi_{x_0\Diamond x_1,w})=
&\begin{pmatrix}\frac{w}{\sigma^2(x_0,x_0)}+\frac{1-w}{\sigma^2(-x_1,x_1)} && 0 && 0 \\ 
0 && \frac{w x_0^2}{\sigma^2(x_0,x_0)}+\frac{(1-w)x_1^2}{\sigma^2(-x_1,x_1)}  &&\frac{wx_0^2}{\sigma^2(x_0,x_0)}-\frac{(1-w)x_1^2}{\sigma^2(-x_1,x_1)}  \\
0 && \frac{wx_0^2}{\sigma^2(x_0,x_0)}-\frac{(1-w)x_1^2}{\sigma^2(-x_1,x_1)} && \frac{w x_0^2}{\sigma^2(x_0,x_0)}+\frac{(1-w)x_1^2}{\sigma^2(-x_1,x_1)} \\
\end{pmatrix}
\end{align*}
with determinant
\[
\det(M(\xi_{x_0\Diamond x_1,w}))=4\frac{(w\, \sigma^2(-x_1,x_1)+(1-w)\, \sigma^2(x_0,x_0))\, w\, (1-w)\, x_0^2\, x_1^2}{\left(\sigma^2(x_0,x_0)\, \sigma^2(-x_1,x_1)\right)^2}.
\]
Maximizing the determinant with respect to the variables $x_0$, $x_1$ and $w$ leads to the following results. Note that $ d_2=0 $ is excluded as this was already settled in \cite{GDHS2012}.

\begin{thm} \label{t:p=3}
In the heteroscedastic model of two-factorial multiple regression on $[-1,1]^2$ with dispersion matrix \emph{(\ref{eq:Dforp=3})} it follows:
\begin{itemize}
\item[(i)] If  $d_0\le d_1-|d_2|, $ the design $\xi_{x_0^*\Diamond x_1^*;0.5}$ is $D$-optimal with 
\[
\begin{tabular} {cc}
$x_0^*=\sqrt{\frac{d_0}{d_1+d_2}}$ & $x_1^*=\sqrt{\frac{d_0}{d_1-d_2}}$ 
\end{tabular}
\]
\item[(ii)] If  $ d_1-|d_2| \le d_0\le \frac{d_1^2-d_2^2}{d_1}$ the design $\xi_{x_0^*\Diamond x_1^*,w^*}$ is $D$-optimal with 
\begin{align*}
&w^*=\frac{2}{3}-\frac{d_0}{6(d_1-d_2)},\quad  x_0^*=\sqrt{\frac{d_1-d_2}{d_1+d_2}\cdot \frac{d_0}{2(d_1-d_2)-d_0}}, \quad x_1^*=1, \quad  \text{if} \quad d_2>0,\\
&w^*=\frac{1}{3}+\frac{d_0}{6(d_1+d_2)},\quad  x_0^*=1,\quad  x_1^*=\sqrt{\frac{d_1+d_2}{d_1-d_2} \cdot \frac{d_0}{2(d_1+d_2)-d_0}},\quad  \text{if} \quad d_2<0
\end{align*}

\item[(iii)] If  $\frac{d_1^2-d_2^2}{d_1}\le d_0 $ the vertex design $\xi_{1\Diamond 1,w^*}$ is $D$-optimal where $w^*$ solves the equation 
\[
2 \left(d_2 \left(6 w ^2-6 w +1\right)+d_1 (1-2 w )\right)+d_0 (1-2 w )=0.
 \]
\end{itemize}
\end{thm}

\begin{proof}
Check with Theorems \ref{t:vertex} and \ref{t:non-vertex}, that the designs are optimal.
\end{proof}
(i) and (ii) describe non-vertex rhombic designs, while (iii) describes the rhombic design with support on the vertices of the square. The Theorem shows that there is a $D$-optimal rhombic design for all $ (d_0,d_1,d_2)^T \in \mathcal{C}_2 $.
Figures \ref{f:p=3} and \ref{f:p=3puzzle} visualize the different optimality regions in Theorem \ref{t:p=3} for $ d_0=1 $. Note that the region only depends on the quotients $ \frac{d_1}{d_0} $ and $ \frac{d_2}{d_0} $, so the choice of $ d_0 $ is arbitrary.

\begin{figure}
\centering
 \subfloat[ Regions (i), (ii)]{\includegraphics[scale=0.5]{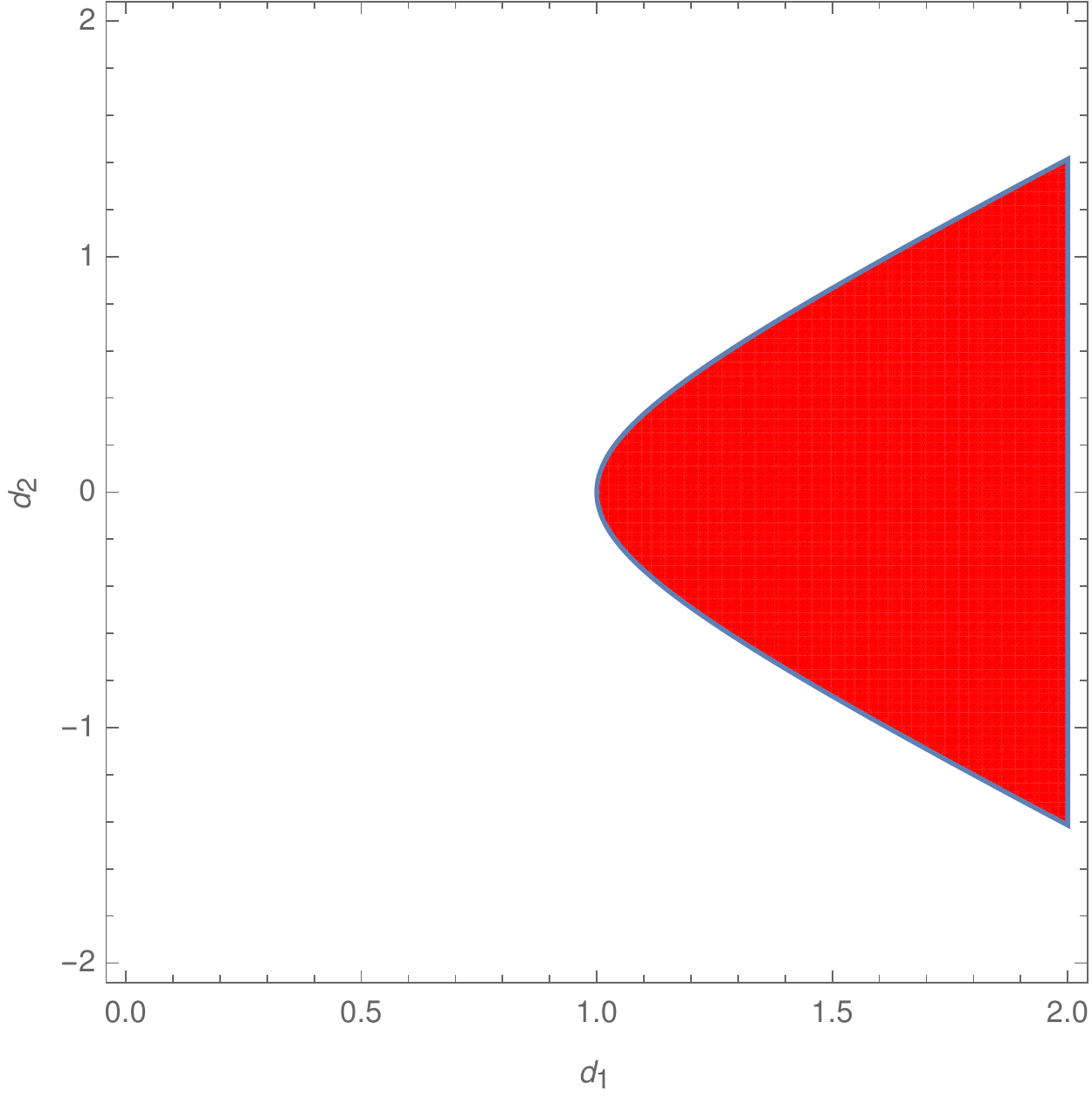}}
~~~~~\subfloat[Region (iii)]{\includegraphics[scale=0.5]{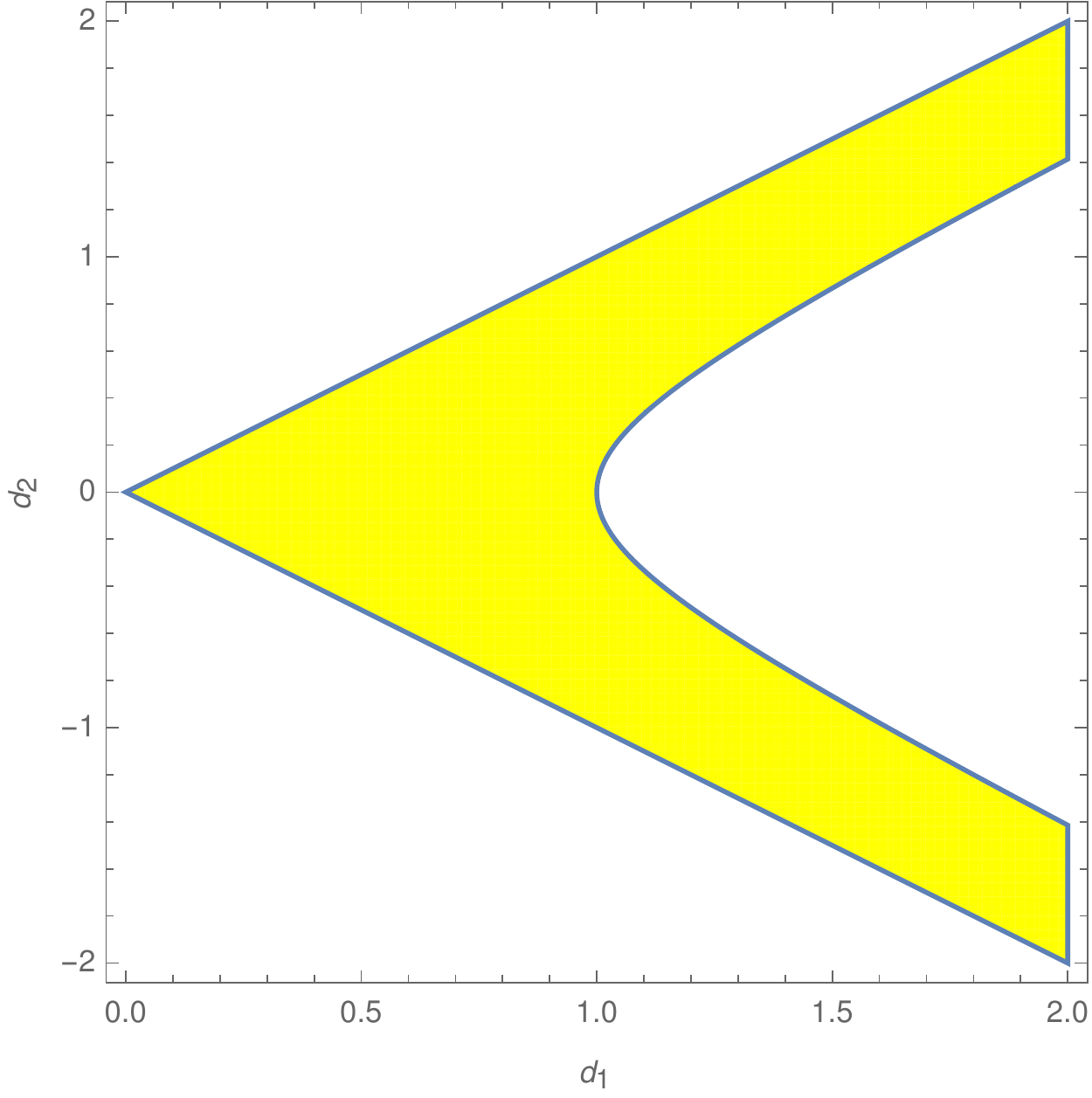}}
\caption{Parameter regions for $ K=2 $: Figure (A) shows parameter regions of rhombic non-vertex designs, while Figure (B) shows parameter regions of rhombic vertex designs }\label{f:p=3}
\end{figure}

\begin{figure} 
\centering
\includegraphics[scale=0.6]{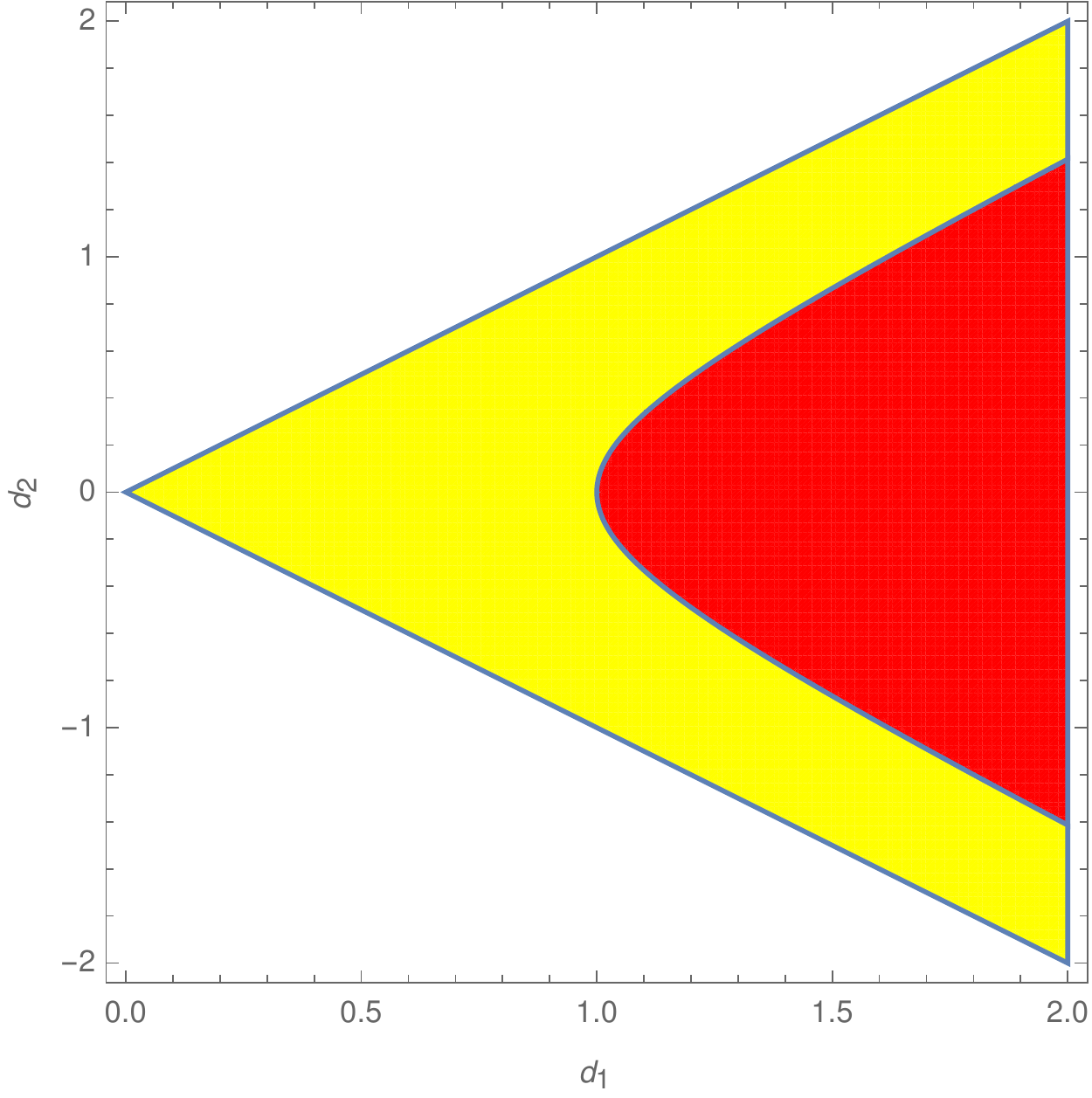}
\caption{Assembling the parameter regions of rhombic designs for $ K=2 $} \label{f:p=3puzzle}
\end{figure}

\subsection{The case \texorpdfstring{$ K=3 $}{K=3}}
The following Theorem results from a \Mathematica implementation of Theorems \ref{t:vertex} and \ref{t:non-vertex}.
\begin{thm}\label{t:p=4}
For the setting from Section \ref{s:Multiplelinregression} with $ K=3 $, so
\[D=\left(\begin{array}{cccc}
d_0   & 0     & 0 &0\\
0     &d_1   &d_{2} &d_2 \\
0&d_{2}&d_1    &d_2 \\
0&d_{2}&d_2&d_1    \\
\end{array}\right),\]
where $ -\frac{d_1}{2}<d_2<d_1 $ it follows:
\begin{itemize}
\item[(i)] If either $ d_0<d_1+d_2 \land 0<d_2<\frac{d_1}{2} $ or $ d_0<\frac{(d_1+2 d_2)^2}{d_1-2 d_2}\land d_2< 0 $, the design with $ w^*=\frac{1}{4} $ and
\begin{align*}
x_0^*&=\frac{\sqrt{d_0 (d_1-2 d_2)}}{d_1+2 d_2},   &&x_1^*=\sqrt{\frac{d_0}{d_1-2d_2}},\\
\end{align*}
is $ D $-optimal.
\item[(ii)] If $ d_2<\frac{d_1}{2}\land d_0<\frac{(d_1-d_2)(d_1+2d_2)}{d_1+d_2} $, it holds that
the design with
\begin{align*}
w^*&=\frac{3 d_0-7 d_1+10 d_2}{-16(d_1-d_2)}, &&x_0^*=\sqrt{\frac{d_0 (-d_1+2 d_2)}{(d_1+2 d_2) (3 d_0-4 d_1+4 d_2)}}, &&& x_1^*=1,
\end{align*}
is $ D $-optimal.
\item[(iii)] If $ d_2<\frac{d_1}{2}\land d_0<\frac{(d_1-d_2)(d_1+2d_2)}{d_1+d_2} $, it holds that
the design with
\begin{align*}
 x_0^*&=1, \qquad  x_1^*=\sqrt{\frac{3d_0 (d_1+2d_2)}{2 d_0 d_2-d_0 d_1-8 d_2^2+4 d_2 d_1+4 d_1^2}},\\
w^*&=\frac{(d_1-2 d_2) (d_0+3 d_1+6 d_2)}{16 (d_1-d_2)(d_1+2d_2)},
\end{align*}
is $ D $-optimal.
\item[(iv)] If $ d_2 \neq 0 $ and $
\left(d_0(d_1+d_2)\geq (d_1-d_2)(d_1+2d_2)\land \ d_2\le \frac{d_1}{2}\right)
\lor (\frac{d_1}{2}<d_2\land 3 d_0+9 d_1>22 d_2), 
$
then the design with $ x_0^*=x_1^*=1 $ and
\begin{multline*}
w^*=\frac{3 d_0^2+22 d_0 d_2+18 d_0 d_1-120 d_2^2+66 d_2 d_1+27 d_1^2}{64 d_2 (d_0-3 d_2+3 d_1)}\\
 -\frac{3 \sqrt{(d_0-2 d_2+3 d_1)^2 \left(d_0^2+8 d_0 d_2+6 d_0 d_1+48 d_2^2+24 d_2 d_1+9 d_1^2\right)}}{64 d_2 (d_0-3 d_2+3 d_1)}
\end{multline*}
is $ D $-optimal.
\end{itemize}
\end{thm}
\begin{proof}
For the cases (i), (ii), (iii) check that the equation $ \frac{1}{4}D^{-1}=M(\xi^*) $ holds and the model constraints are satisfied.
For the fourth case, check that $ m_0(\xi^*)\ge \frac{1}{4d_0} $ and that the model constraints are satisfied.
\end{proof}
Note that not all settings of $ (d_0,d_1,d_2) $ are covered by Theorem \ref{t:p=4} and that the described design areas are not disjoint. (ii) and (iii) describe the same optimality area that also contain area (i). Figures \ref{f:p=4} and \ref{f:p=4puzzle} show the optimality area for $ d_0=1 $ in the $ (d_1,d_2) $-space. Again, the region only depends on the quotients $ \frac{d_1}{d_0} $ and $ \frac{d_2}{d_0} $, so the choice of $ d_0 $ is arbitrary. The area where we did not find an optimal rhombic design is given by $ \frac{d_1}{2}<d_2\land 3 d_0+9 d_1\le 22 d_2 $.

\begin{figure}
\centering
\subfloat[Regions (i), (ii), (iii)]{\includegraphics[scale=0.5]{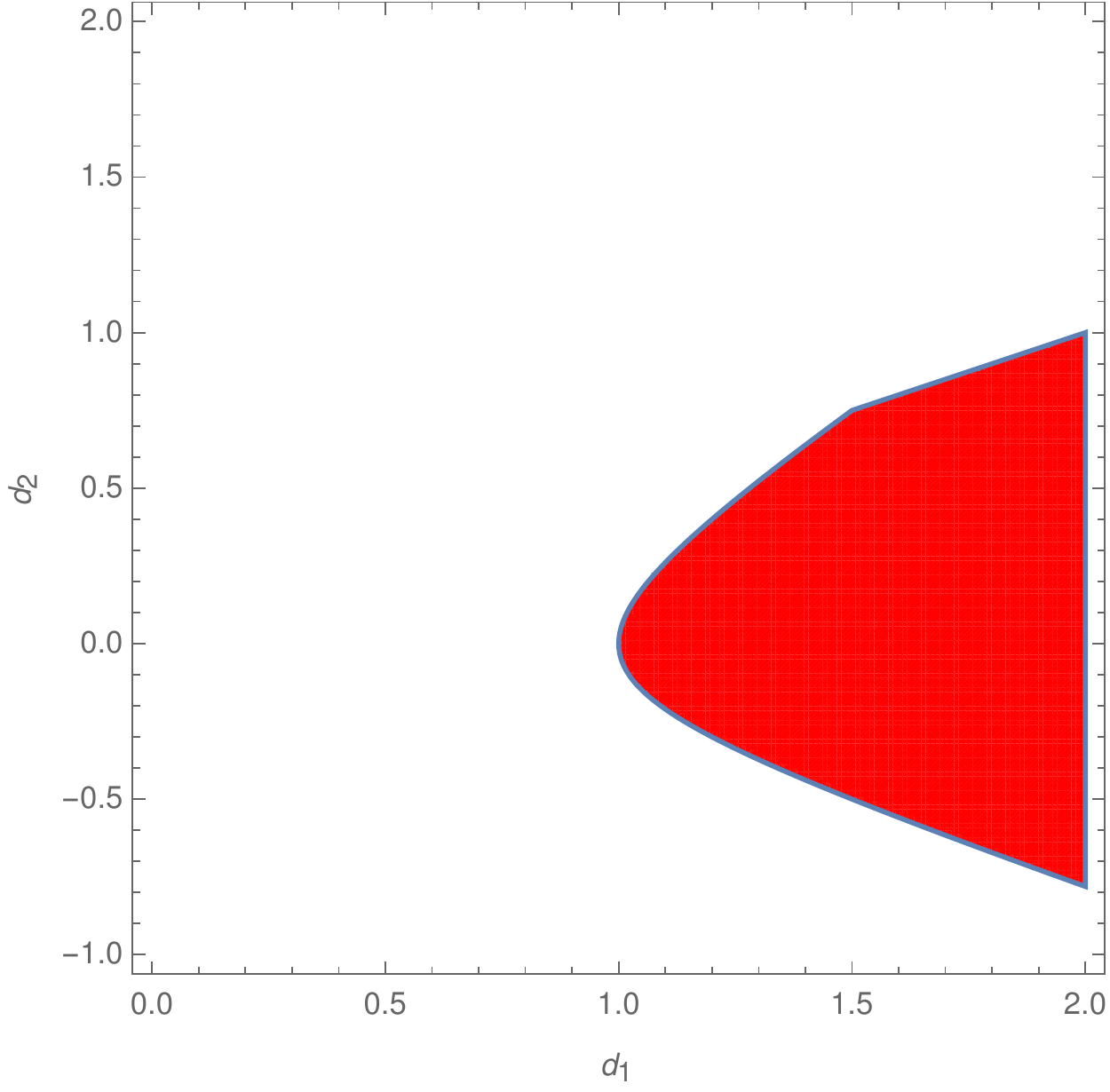}}
~~~~\subfloat[Region (iv)]{\includegraphics[scale=0.5]{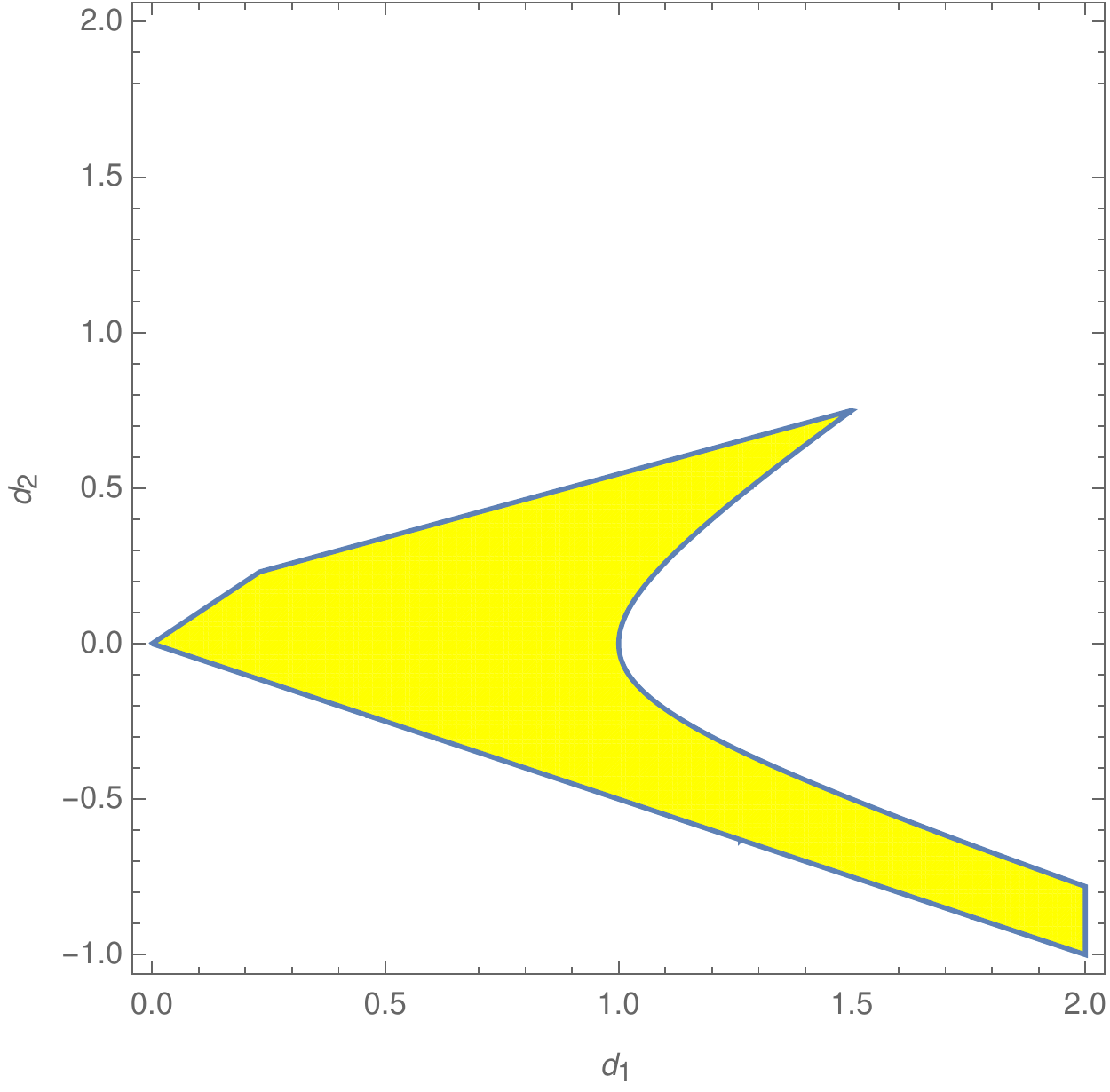}}
\caption{Parameter regions for $ K=3 $: Figure (A) shows parameter regions of rhombic non-vertex designs, while Figure (B) shows parameter regions of rhombic vertex designs}\label{f:p=4}
\end{figure}

\begin{figure}
\centering
\includegraphics[scale=0.6]{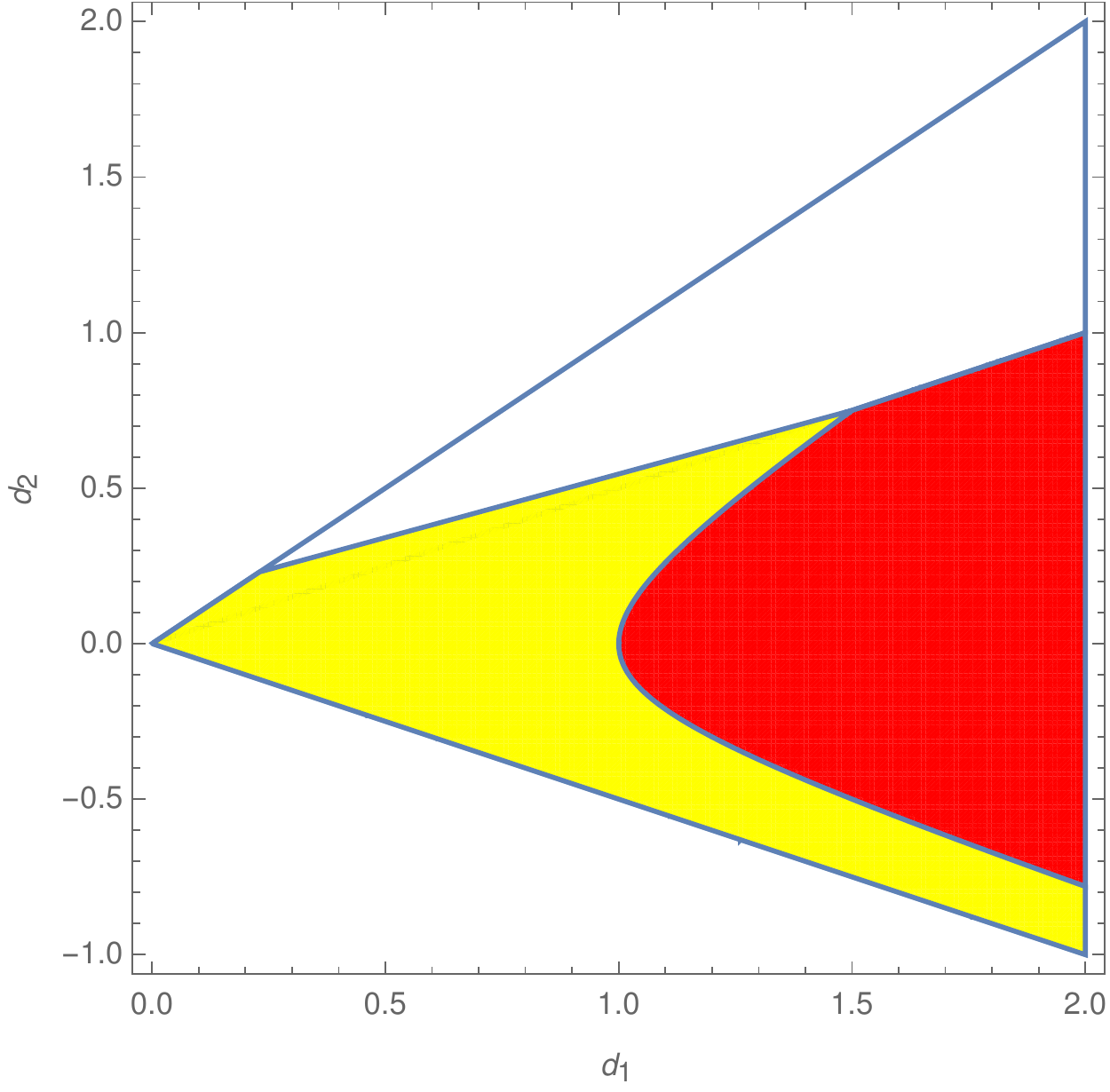}
\caption{Assembling the parameter regions for $ K=3 $}\label{f:p=4puzzle}
\end{figure}

\subsection{The cases \texorpdfstring{$ K=4 $}{K=4} and \texorpdfstring{$ K=5 $}{K=5}} 
For $ K=4 $ and $ K=5 $ there are up to three orbits for rhombic designs. To compute an optimal rhombic vertex design, we let $ \mathcal{O}_\ell(x_\ell) $ denote the orbits of rhombic design points and choose $ x_0=x_1=x_2=1 $, such that the weights are $ w_\ell=\xi(\mathcal{O}_\ell(1)) $ and check the conditions in Theorem \ref{t:vertex} for optimality. 
The different optimality areas are shown in Figure \ref{f:p=5puzzle}. Again, in the red region, an rhombic design with interior points is $ D $-optimal, while in the yellow area, a rhombic vertex design is $ D $-optimal. The separating line is again given by the equality of the first and the second diagonal entry of $ D^{-1} $, see Corollary \ref{c:vertex} and Corollary \ref{c:non-vertex}. We see a similar structure as for $ K=2 $ and $ K=3 $. For $ K=4 $, there is a $ D $-optimal rhombic design for every point in $ \mathcal{C}_4 $, while for $ K=5 $, in the region above $ d_2=\frac{d_1}{2} $ there is only a small area where rhombic designs are $ D $-optimal, similar to the case for $ K=3 $.

	\begin{figure}
		\centering
		\subfloat[Parameter regions for $ K=4 $]{\includegraphics[scale=.45]{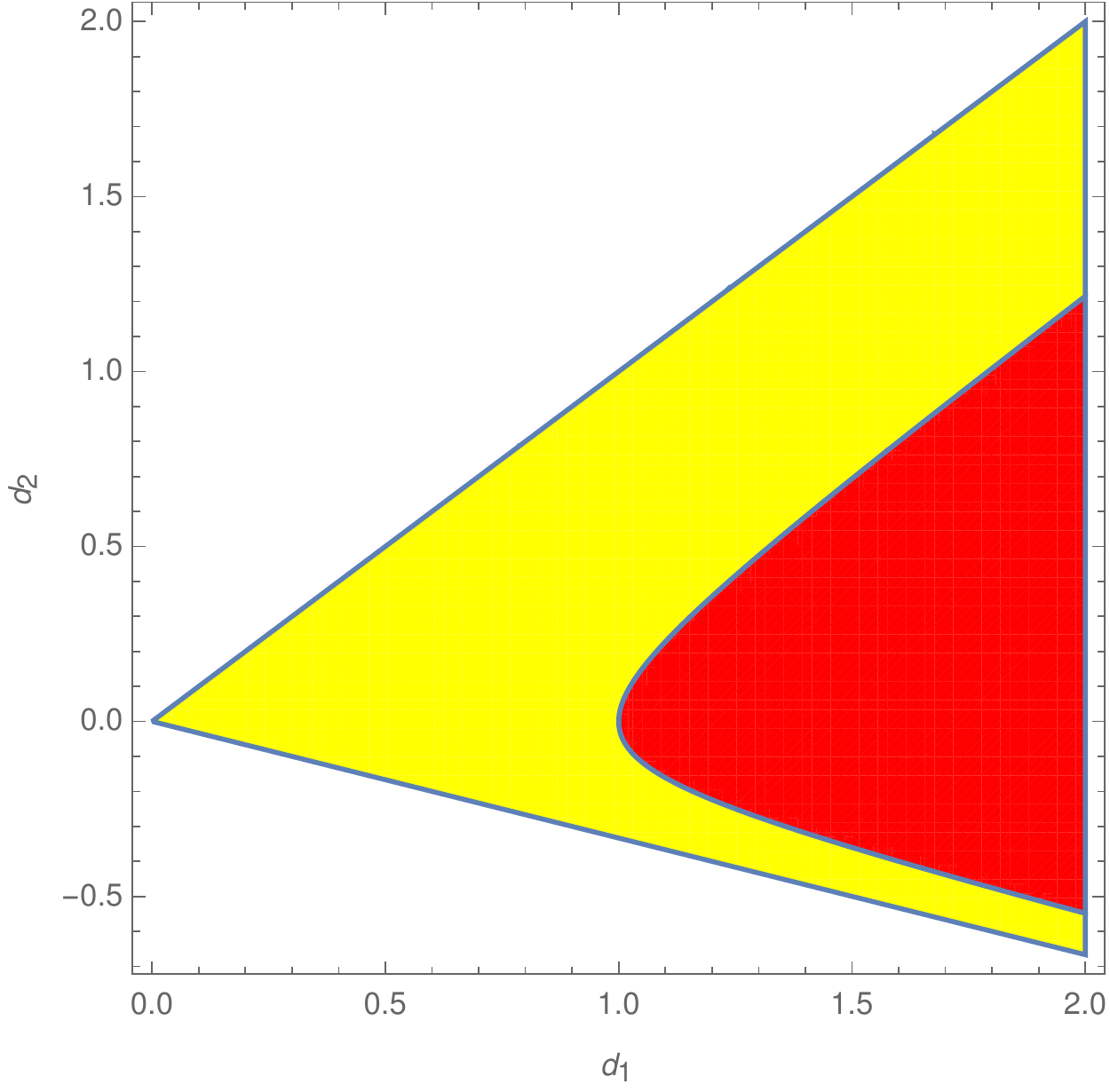}}
		~~
		\subfloat[Parameter regions for $ K=5 $]{	\includegraphics[scale=.45]{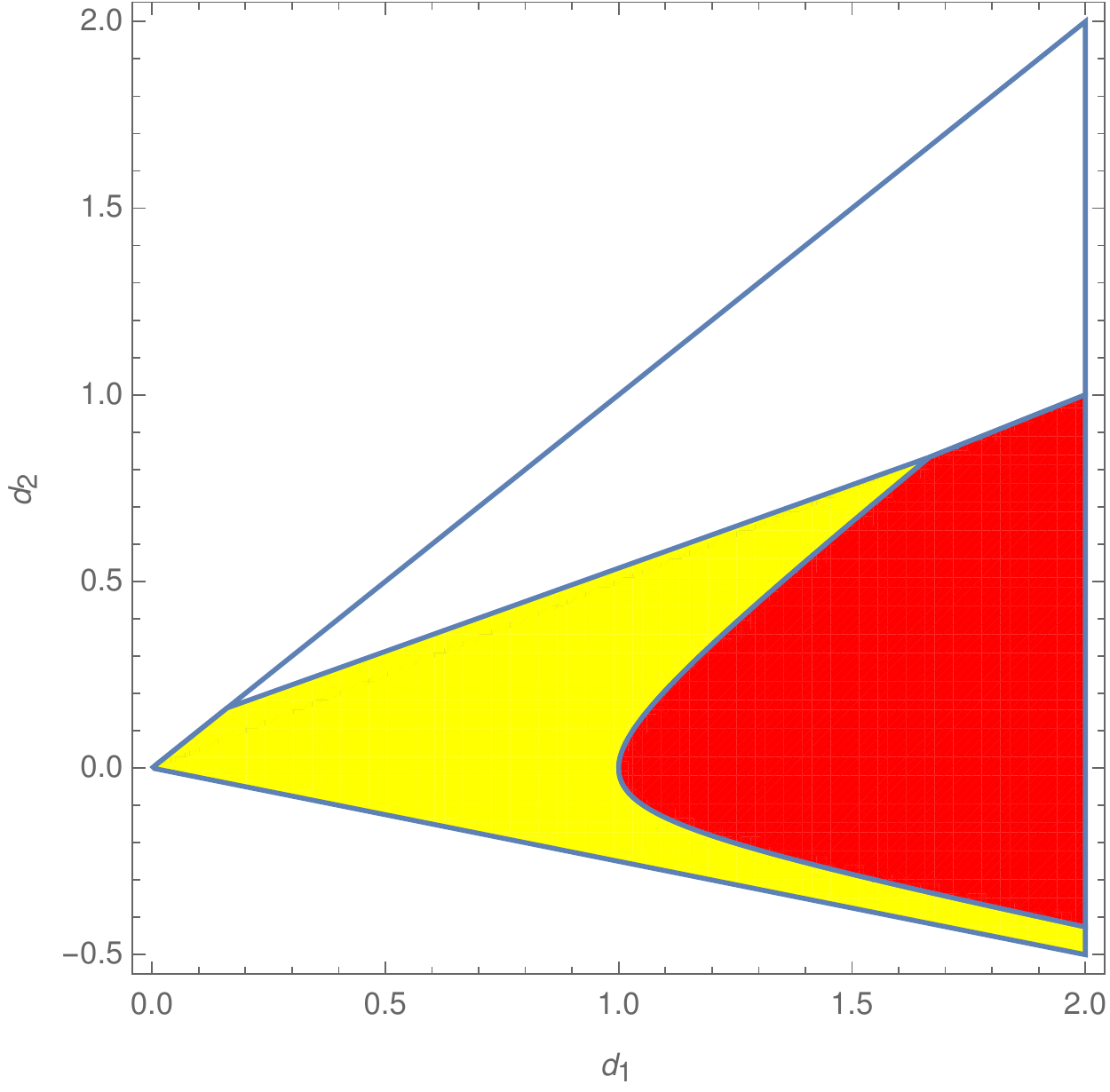}}
\caption{Assembling the optimality regions for $ K=4 $ and $ K=5 $} \label{f:p=5puzzle}
	\end{figure}

\begin{remark}
The optimality regions shown in the figures for $ K\in \{2,3,4,5\} $ are given in the $ (d_1,d_2) $-space while $ d_0=1 $. As before, the region only depends on the quotients $ \frac{d_1}{d_0} $ and $ \frac{d_2}{d_0} $, so the choice of $ d_0=1 $ is arbitrary. $ D $-optimal designs and the corresponding parameter regions where they are optimal can be found by studying the semi-algebraic sets as described in Remark \ref{r:non-vertex} and Remark \ref{r:vertex}. A convenient way to generate the images showing the optimality regions is therefore to use the \textbf{Resolve} and \textbf{RegionPlot} commands of \Mathematica to compute and plot these regions. This was done for $ K\in \{2,3,4,5\} $.
\end{remark}

\section{Discussion} \label{s:conjectures}

In the preceding sections optimality regions have been investigated for certain invariant designs in a multiple linear regression model on the hypercube with invariant correlation structure of the random coefficients. It has been shown that for the introduced class of rhombic designs, it is possible to decide whether a $D$-optimal design is either supported on the vertices of the hypercube or has interior design points by evaluating a quadratic polynomial depending on the covariance matrix of the random coefficients. This result relies on the Kiefer-Wolfowitz equivalence theorem.
The equation separating the two optimality regions is given as the equality of the diagonal entries of $ D^{-1} $.

The results of Theorem \ref{t:non-vertex} hold not only for rhombic designs but for all designs with an interior design point, independently of invariance considerations. This means that the $D$-optimality of designs with interior points is equivalent to the equation $ M(\xi^*)=\frac{1}{p}D^{-1} $.

An important observation is the apparent non-existence for $D$-optimal rhombic designs for certain values of the entries $ D $. For small dimensions, we have observed that for even $ K $, we could always find a $D$-optimal rhombic design for any $ D $, while this has not been true for odd $ K $. With respect to our findings, we conjecture the following:
\begin{conj}
For even $ K $, there is a $ D $-optimal rhombic design for all $ (d_0,d_1,d_2)^T \in \mathcal{C}_K $.	
For odd $ K $, there is a $ D $-optimal rhombic design for all $ (d_0,d_1,d_2)^T \in \mathcal{C}_K $ with $ d_2\le \frac{d_1}{2} $.
\end{conj}

\subsection*{Acknowledgements}
Work supported by grants HO\,1286/6, SCHW\,531/15 and 314838170, GRK 2297 MathCoRe of the Deutsche Forschungsgemeinschaft DFG.

\bibliographystyle{plain}
\bibliography{bibliography}

\end{document}